\documentclass[final]{siamltex}
\usepackage{amssymb,amsmath}

\def\R{\mathbb{R}}
\def\N{\mathbb{N}}
\def\Z{\mathbb{Z}}
\def\u{{\bf u}}
\def\v{{\bf v}}
\def\w{{\bf w}}
\def\bbbn{\mathbb{N}}

\def\Xi{\chi^*_{[a_i, a_{i+1}]}}
\newcommand{\eps}{\varepsilon}
\newcommand{\ho}{{\mathcal H}^{0}}

\newtheorem{remark}[theorem]{Remark}

\title{A chain rule formula in BV and applications to conservation laws}

\author{Graziano Crasta\thanks{Dipartimento di Matematica 
``G.\ Castelnuovo'', Univ.\ di Roma I,
P.le A.\ Moro 2, Roma, Italy  I-00185
(\tt crasta@mat.uniroma1.it).}
\and Virginia De Cicco\thanks{
Dipartimento di Scienze di Base  e Applicate per l'Ingegneria,
Via A.\ Scarpa 10, Roma, Italy I-00185
(\tt decicco@dmmm.uniroma1.it)}}

\begin{document}

\maketitle

\begin{abstract}
In this paper we prove a new chain rule formula for the distributional derivative of the composite function $v(x)=B(x,u(x))$, where $u:]a,b[\to\R^d$ has bounded variation, $B(x,\cdot)$ is continuously differentiable and $B(\cdot,u)$ has bounded variation.
We propose an application of this formula in order to deal in an intrinsic way 
with the discontinuous flux appearing in conservation laws in one space variable.
\end{abstract}

\begin{keywords} 
Chain rule, $BV$ functions, conservation laws with discontinuous flux
\end{keywords}

\begin{AMS}
Primary: 26A45, 35L65; Secondary: 26A24, 46F10
\end{AMS}

\pagestyle{myheadings}
\thispagestyle{plain}
\markboth{G. Crasta and V. De Cicco}{A chain rule formula in BV}

\section{Introduction}
In 1967, A.I. Vol'pert in \cite{vol} (see also \cite{vol1}), in view of applications in the study of quasilinear hyperbolic equations,  established a chain rule formula for distributional derivatives of the composite function $v(x)=B(u(x))$\,, where $u:\Omega\to\R$ has bounded variation in the open subset $\Omega$ of $\R^N$ and $B:\R\to\R$ is continuously differentiable. He proved that $v$ has bounded variation and its distributional derivative 
$Dv$ (which is a Radon measure on $\Omega$) admits an explicit representation in terms of the gradient $\nabla B$ and of the distributional derivative $Du$\,.
More precisely, the following identity holds in the sense of measures: 
\begin{equation}\label{chainDv}
Dv=\nabla B(u)\nabla u \ \mathcal L^{N}+\nabla B(\widetilde u) D^cu+[B(u^+)-B(u^-)]\, \nu_u\, \mathcal H^{N-1}
\lfloor{J_{u}}\,,
\end{equation}
where
\begin{equation}\label{decompDu}
Du=\nabla u \ \mathcal L^{N}+D^cu+\nu_u\, \mathcal H^{N-1}
\lfloor{J_{u}}\,
\end{equation}
is the usual decomposition of $Du$ in its absolutely continuous part $\nabla u$ with respect to the Lebesgue measure $\mathcal L^{N}$, its Cantor part $D^cu$ and its jumping part, which is  
represented by the restriction of the $(N-1)$-dimensional Hausdorff measure to the jump set $J_u$\,. Moreover, $\nu_u$ denotes the measure theoretical unit normal to $J_u$, $\widetilde u$ is the approximate limit and $u^+$, $u^-$ are the approximate limits from both sides of $J_u$\,.

The validity of (\ref{chainDv}) is stated also in the vectorial case (see \cite{adm} and Theorem 3.96 in \cite{AFP}), namely if $u:\Omega\to\R^d$ has bounded variation  and $B:\R^d\to\R$ is continuously differentiable, then 
the terms in (\ref{chainDv}) should be interpreted in the following sense:
\begin{equation}\label{chainDvvect}
Dv=\nabla B(u)\cdot \nabla u\ \mathcal L^{N}+\nabla B(\widetilde u)\cdot D^cu+
[B(u^+)-B(u^-)]\otimes\nu_u\, \mathcal H^{N-1}
\lfloor{J_{u}}\,.
\end{equation}
The situation is significantly more complicated if $B$ is only a Lipschitz continuous function. In this case, the general chain rule is false, while a weaker form of the formula was proved by Ambrosio and Dal Maso in  \cite{adm} (see also \cite{lm}). 

On the other hand, in some recent papers a remarkable effort is devoted to establish chain rule formulas with an explicit dependence on the space variable $x$\,. This amounts to describe the
distributional derivative of the composite function $v(x)=B(x,u(x))$, where $B(x,\cdot)$ is continuously differentiable and, for every $s\in \R^d$, $B(\cdot,s)$ and $u$ are functions with low regularity (which will be specified later).
These formulas have applications, for example,
in the study of the $L^1$ lower semicontinuity of approximating linear 
integrals of convex non-autonomous functionals (see \cite{dcfv1}, \cite{dcfv2} and \cite{adcmm}).

The first formula of this type is established in \cite{dcl} for functions $u\in W^{1,1}(\Omega;\R^d)$ by assuming that, for every $s\in \R^d$, $B(\cdot,s)$ is an $L^1$ function whose distributional divergence belongs to $L^1$ (in particular it holds if $B(\cdot,s)\in W^{1,1}(\Omega;\R^d)$\,). 

In \cite{dcfv2} the formula is proved by assuming that,  for every $s\in \R^d$, $B(\cdot,s)$ is an $L^1$ function whose distributional divergence is a Radon measure with bounded total variation and  $u\in W^{1,1}(\Omega;\R)$\,.

The case of a function $u\in BV(\Omega)$ is studied in the papers \cite{dcfv1} and \cite{dcfv2}. In the first paper the authors have established the validity of the chain rule by requiring that $B(\cdot,s)$ is differentiable in the weak sense for every $s\in \R$. In the second one it is assumed only a $BV$ dependence of
$B$ with respect to the variable $x$\,.

The main difficulty of these results consists in giving sense to the different terms of the formula. 
Notice that the new term of derivation with respect to $x$ needs a particular attention. 
For instance in \cite{dcfv2} this term is described by a Fubini's type inversion of integration order.

\bigskip
The aim of this paper is to establish a chain rule formula 
for the distributional derivative of the composite function $v(x)=B(x,u(x))$\,, 
where $u:]a,b[\to\R^d$ has bounded variation, 
$B(x,\cdot)$ is continuously differentiable and $B(\cdot,s)$ has bounded variation.
We assume that
there exists a countable set $\mathcal N\subset ]a,b[$ such that 
the jump set $J_{B(\cdot,s)}$ of ${B(\cdot,s)}$
is contained in $\mathcal{N}$ for every $s\in\R^{d}$.
Moreover we require that there exists a positive finite Cantor measure $\lambda$ on $]a,b[$ such that $(D^{c}_x B)(\cdot,s) \ll \lambda$ for every $s\in \R^d$\,.
For every $s\in \R^d$ let $\psi(\cdot,s)$ denote the Radon-Nikod\'ym derivative of the measure $(D^{c}_x B)(\cdot,s) $ with respect to $\lambda$, i.e.
$$
\psi(\cdot,s):= \frac{d(D^{c}_x B)(\cdot,s)}{d\lambda}\,.
$$

We show that (see Theorem \ref{chain} below), 
under suitable additional assumptions, the composite function $v(x):=B\left(x,u(x)\right)$ 
belongs to $BV(]a,b[)$
and for any
$\phi\in C_0^1(]a,b[)$ we have
\begin{equation}\label{f:intro}
\begin{split}
\int_{]a,b[} \phi'(x)  v(x)\,dx = {} &
-\int_{]a,b[} \phi(x) (\nabla_x B)(x,u(x))\, dx
\\ & -
\int_{]a,b[} \phi(x) \psi(x,u(x))\, d\lambda
\\ & -
\int_{]a,b[} \!\phi(x) (D_{s}B)(x,u(x))\cdot\nabla u(x)\,dx
\\ & -
\int_{]a,b[}\phi(x)({D_{s}B})(x,u(x))\,\cdot dD^cu(x)
\\ & -
\sum_{x\in  \mathcal N\cup J_u}\!\!\phi(x)
\left[B(x_{+},u(x_{+}))-B(x_{-},u(x_{-}))\right]\,,
\end{split}
\end{equation}
where $u(x_{+})$, $u(x_{-})$ and $B(x_{+},s)$, $B(x_{+},s)$ are 
respectively the right and left limits of $u$ and $B(\cdot,s)$ at $x$\,.

The proof is based on a regularization argument via convolutions and 
on the Ambrosio-Dal Maso derivation formula (see \cite{adm}).
In order to prove the convergence of the regularized terms
we follow the arguments as in \cite{dcfv2}, 
with the exception of the term of derivation with respect to $x$,
which requires a different nontrivial analysis
due to the possible interaction of the jump points of $u$ 
and the jump points 
of $B(\cdot,s)$.

In order to understand this effect, we consider firstly a piecewice constant function $u$, and we show that, in this case, the contributions of the jump parts can be collected as in the summation in (\ref{f:intro}).
The general case can be obtained by using a precise approximation result, proven in Section 3, of a $BV$ function by piecewise constant functions which holds only for functions defined on an interval. By the way, we remark that this is one of the technical point where it is crucial the restriction to a one dimensional space variable. 

In Section 5, we consider the case $d=1$ and we compare our chain rule with the formula proven in \cite{dcfv2}. We verify the (necessary!) coincidence of the terms of derivation with respect to $x$ in the case of piecewise constant functions $u$. 
Anyway, we remark that the form (\ref{f:intro}) is new also in this one-dimensional case.

\bigskip
Finally,
in Section 6 we discuss the use of our chain rule formula
to conservation laws with a discontinuous flux.
The case of discontinuous fluxes has been intensively studied in
the last few years (see e.g.\ 
\cite{AP,disc0,disc2,disc3,disc1,disc4,GNPT,disc6,disc5,disc8,disc7,disc9} 
and the references therein)
due to a large class of applications in physical and traffic models. 

We do not address directly the issue of existence or uniqueness of solutions, for which we refer to the references listed above.
We remark that the existence results are proved only for very special fluxes (tipically, only one jump in the space variable is allowed).
For what concerns uniqueness, we recall a fairly general result by Audusse and Perthame 
\cite{AP}, which is based on an extension of the classical Kruzkov method.

In this framework, using our chain rule formula, we propose a definition of entropic solution which is a generalization of the classical one valid for smooth fluxes (see e.g. \cite{Daf}).

We show that our definition is equivalent, under suitable assumptions,
to the notion of 
Kruzkov--type entropic solution obtained using
the adapted entropies introduced by Audusse and Perthame in \cite{AP}.
Our formula provides a neat environment for the treatment
of all terms containing a derivative of the composition with a $BV$ function
which are present in equations of this type.

We are inclined to believe that the methods here introduced can be useful to treat analogous 
problems in the same context.

\bigskip
\textbf{Acknowledgements.}
The authors would like to thank Gianni Dal Maso and Nicola Fusco
for stimulating discussions and suggestions during the preparation
of the manuscript.

\section{BV functions of one variable}
In this section we introduce the $BV$ functions of one variable and we recall the definitions and the basic results (see the book \cite{AFP} for a general survey on this subject).

We recall that a function $\u=(u^1,\dots,u^d)\in L^1(]a,b[;\R^d)$ belongs to 
the space $BV(]a,b[;\R^d)$ if and only if
\begin{equation}\label{definBV}
TV(\u):=\sup\Big\{\sum_{i=1}^{d}\int_a^b u^i D\phi^i\, dx: \phi\in C^1(]a,b[;\R^d)\,, \|\phi\|_\infty\leq 1\Big\}<+\infty
\end{equation}
(if $d=1$ the usual notation is  $BV(]a,b[$)\,).
This implies that the distributional derivative $D\u=(Du^1,\dots,Du^d)$ is a bounded Radon measure in $]a,b[$ and the following integration by parts formula holds:
\begin{equation}\label{parts}
\int_a^b u^i D\phi^i\, dx=-\int_a^b \phi^i\,dDu^i\quad\quad \forall \phi\in C^1(]a,b[;\R^d),\quad i=1,\dots,d\,.
\end{equation}
A measure $\mu$ is absolutely continuous with respect to a positive measure $\lambda$ ($\mu\ll\lambda$ in symbols) if
$\mu(B)=0$ for every measurable set $B$ such that $\lambda(B)=0$\,.
We will often consider the Lebesgue decomposition
\begin{equation}\label{decomp}
D\u=\nabla \u\,dx+D^s\u\,,
\end{equation}
where $\nabla \u$ denotes the density of the {\it absolutely continuous part} of $D\u$ with respect to the Lebesgue measure on $]a,b[$\,, while $D^s\u$ is its {\it singular part}.

For every function $\u\in BV(]a,b[;\R^d)$ the following left and right limits
\begin{equation}\label{apprlim}
\u(x_-):=\lim_{\eps\to 0+}\frac{1}{\eps}\int_{x-\eps}^{x}\u(y)\,dy\,,
\quad\quad
\u(x_+):=\lim_{\eps\to 0+}\frac{1}{\eps}\int_x^{x+\eps}\u(y)\,dy
\end{equation}
exist at every point $x\in ]a,b[$\,. 
In fact, $\u(x_-)$ is well defined also in $x=b$, while $\u(x_+)$ exists  also in $x=a$.
The left and right limits just defined coincide a.e. with $\u$  and are left and right continuous, respectively.

It is well known that the \textit{jump set} of $\u$, defined by
\[
J_{\u}:=\{x\in ]a,b[: \u(x_-)\not=\u(x_+)\}
\]
is at most countable.
The singular part $D^s\u$ of the measure $D\u$ can be splitted into the sum of a  measure concentrated on $J_{\u}$ and a measure $D^c\u$, called the {\it Cantor part} of $D\u$, as in the following formula:
\begin{equation}\label{decomp1}
D^{s}\u=D^c\u+\big(\u(x_+)-\u(x_-)\big)\mathcal{H}^{0}\lfloor{J_{\u}}\,,
\end{equation}
where $\mathcal H^{0}$ stands for the counting measure. Moreover, we consider the so-called {\it diffuse part} of the measure $D\u$ concentrated on $C_{\u}:=]a,b[\setminus J_{\u}$ and defined by
\begin{equation}\label{diffusepart}
\widetilde D\u:=\nabla \u\,dx+D^c\u
\,,
\end{equation}
while
\begin{equation}\label{atomicpart}
D^j\u:=\big(\u(x_+)-\u(x_-)\big)\mathcal{H}^{0}\lfloor{J_{\u}}
\end{equation}
is called the {\it atomic part} of $D\u$\,. Analogously, we said that a nonnegative Borel measure $\mu$ is a {\it Cantor measure} if $\mu$ is a diffuse measure orthogonal to the Lebesgue measure.

If $|D\u|$ denotes the total variation measure of $D\u$, we have that $|D\u|(]a,b[)$ equals the value of the supremum in (\ref{definBV}); 
moreover, for every Borel subset $B$ of $]a,b[$,
\begin{equation}\label{decomp2}
|D\u|(B)=\int_{B}|\nabla\u|(x)\,dx+|D^{c}\u|(B)+\sum_{x\in J_{\u}\cap B}|\u(x_{+})-\u(x_{-})|\,.
\end{equation}

Now we recall the classical definition for $BV$ functions of one variable, by means of the {\it pointwise variation}; for every function $\u:]a,b[\to\R^{d}$, it is defined by
\begin{equation}\label{pointvar}
pV(\u):=\sup\Big\{\sum_{i=1}^{n-1}|\u(t_{i+1})-\u(t_{i})|: a<t_{1}<\dots<t_{n}<b\Big\}.
\end{equation}
We remark that every function $\u$ having finite pointwise variation 
belongs to the space $L^{\infty}(]a,b[;\R^d)$,
since its oscillation is controlled by $pV(\u)$.
Moreover every bounded monotone real valued function has finite pointwise variation and any (real valued) function having finite pointwise variation can be splitted into the difference of two monotone functions.

In order to avoid that $\u$ changes if it is  modified even at a single point, we introduced the following definition of {\it essential variation}
\begin{equation}\label{essvar}
eV(\u):=\sup\Big\{pV(\v): \v=\u {\rm \  \ a.e. \  in \  } ]a,b[\Big\}\,.
\end{equation}
Finally, by Theorem 3.27 in \cite{AFP}, the essential variation $eV(\u)$ coincides with the variation $V(\u)$, defined in (\ref{definBV}). Any function $\overline \u$ in the equivalence class of $\u$ (that is $\u=\overline\u$ a.e.) such that
$pV(\overline\u)=eV(\u)=TV(\u)$
is called a {\it good representative}. 
By Theorem 3.28 in \cite{AFP}, we have that $\overline \u$ is a good representative if and only if for every $x\in ]a,b[$
\begin{equation}\label{f:jumpseg}
\overline \u(x)\in \Big\{\theta \u(x_{-})+(1-\theta)\u(x_{+}): \theta\in [0,1]
\Big\}\,.
\end{equation}
In particular, 
if (\ref{f:jumpseg}) holds with $\theta=0$ (resp.\ $\theta=1$)
for every $x\in ]a,b[$, we have that
$\overline{u} = u^+$ (resp.\ $\overline{u} = u^+$),
while for $\theta = 1/2$
$\overline{u}$ coincides with the so-called {\it precise representative}
\begin{equation}\label{precise}
\u^*(x):=\frac{\u(x_+)+\u(x_-)}{2}\,.
\end{equation}
Any good representative $\overline \u$ is continuous in $]a,b[\setminus J_{\u}$,
and it has a jump discontinuity at any point of $J_{\u}$  satisfying
$\overline{\u}(x_{-})=\u(x_{-})$, $\overline{\u}(x_{+})=\u(x_{+})$.
Finally, any good representative $\overline \u$ is a.e.\ differentiable in $]a,b[$ and its derivative $\nabla\u$ coincides with the density of $D\u$ with respect to the Lebesgue measure.
If not otherwise stated,
in this paper we always consider good representatives of $BV$ functions.

For every scalar $BV$ function $u$ the following {\it coarea formula} holds (see \cite{Fed}, Theorem 4.5.9):
\begin{equation}\label{coarea}
\int_a^b g(x) \,d|Du|(x)=\int_{-\infty}^{+\infty}dt\,\int_{\{u(x_-)\leq t\leq u(x_+) \}} g(x)\,d\mathcal H^0(x)
\end{equation}
for every Borel function $g:]a,b[\to[0,+\infty[$.

We remark that a Leibnitz rule formula in $BV(]a,b[)$ holds: if $v,w\in BV(]a,b[)$, then $vw\in BV(]a,b[)$ and
\begin{equation}\label{f:Leib}
D(vw) = v^* Dw + w^*Dv,
\end{equation}
in the sense of measures (see Example 3.97 in \cite{AFP} and Remark 3.3 in  \cite{dcfv2}).

Now we recall the properties of the convolution of a $BV$ function. Let $\varphi$ be a standard convolution kernel and
let
$(\varphi_{\varepsilon})_{\varepsilon>0}$ be a family of
mollifiers, i.e. $\varphi_\eps(x):=\eps^{-1}\varphi(x/\eps)$.  
For every function $\u\in BV(]a,b[;\R^{d})$ we define
\[
\u_{\eps}(x):=(\u*\varphi_\eps)(x)=
\int_{a}^{b}\varphi_{\varepsilon}(x-y)\,\u(y)\,dy
\]
for $x\in]a',b'[\subset\subset]a,b[$
and $0<\varepsilon<\min(b-b',a'-a)$. 
We have that the mollified functions $\u*\varphi_\eps$ converge a.e.\ to $\u$ in $]a,b[$ and 
everywhere in $[a,b[$ to the precise representative $\u^*$ (see Proposition 3.64(b) and Corollary 3.80 in \cite{AFP}).
Moreover
$\nabla \u_\eps=\nabla (\u*\varphi_\eps)=(D\u)*\varphi_\eps$
(see Proposition 3.2 in \cite{AFP}),
where for a Radon measure $\mu$, the convolution $\mu*\varphi_\eps$ is defined as
\[
(\mu*\varphi_\eps)(x):=
\int_{a}^{b}\varphi_{\varepsilon}(x-y)\,d\mu(y)\,.
\]
Finally, we recall that the measures $\nabla \u_\eps\,dx$ locally weakly$^*$ converge in $]a,b[$ to the measure $D\u$, i.e. for every $\phi\in C_0(]a,b[)$ we have
\[
\int_a^b \phi\nabla \u_\eps\,dx\quad\to\quad\int_a^b \phi\,dD\u\,,\quad{\rm as}\ \eps\to 0
\]
(see Theorem 2.2 in \cite{AFP}).

\section{An approximation result}
In this section we exhibit an explicit piecewise constant approximation of a $BV$ function, which is taylored to our needs in the proof of Theorem \ref{chain}\,.

\begin{lemma}\label{l:approx1}
Let $v\in BV(]a,b[)$, let $J$ denote its jump set,
and let $P\subset ]a,b[\setminus J$ be a countable set.
Then, for every $\eps >0$ and every finite set $P_{\eps}\subset P$
there exists a piecewise constant function
$v_{\eps}\colon ]a,b[\to\R$ such that:
\begin{itemize}
\item[(i)]
the (finite) jump set $J_{\eps}$ of $v_{\eps}$ contains all jumps of $v$
of size greater than $\eps/3$;
\item[(ii)]
$TV(v_{\eps}) \leq TV(v)$;
\item[(iii)]
$J_{\eps}\cap P = \emptyset$ and $v_{\eps}(x) = v(x)$ for every $x\in P_{\eps}$;
\item[(iv)]
$v_{\eps}(x_+) = v(x_+)$, $v_{\eps}(x_-) = v(x_-)$, for every $x\in J\cap J_{\eps}$;
\item[(v)]
$|v_{\eps}(x) - v(x)| < \eps$ for every $x\in ]a,b[\setminus J$
(the inequality holds everywhere if $v$ is a good representative).
\end{itemize}
\end{lemma}

\begin{proof}
Without loss of generality we can assume that
$v$ is a good representative.
Let $J=\{x^j\}$ be the jump set of $v$. Since $v\in BV$,
there exists $N\in \bbbn$ such that
\begin{equation}\label{f:resto}
\sum_{j>N} |v(x^j_+) - v(x^j_-)| \leq \frac{\eps}{3}\, .
\end{equation}
Let us define the functions $v_B, v_S\colon ]a,b[\to\R$ by
\[
\begin{split}
v_B(x) :=
\begin{cases}
\sum\limits_{x_j < x, j\leq N}
[v(x^j_+)-v(x^j_-)],
&\text{if $x\not\in \{x^1,\ldots,x^N\}$},\\
v_B(x^j_-) + v(x^j) - v(x^j_-),
&\text{if $x = x^j$ for some $j\leq N$},
\end{cases}
\\
v_S(x) :=
\begin{cases}
\sum\limits_{x^j < x, j > N}
[v(x^j_+)-v(x^j_-)],
&\text{if $x\not\in \{x^i: i>N\}$},\\
v_S(x^j_-) + v(x^j) - v(x^j_-),
&\text{if $x = x^j$ for some $j > N$}.
\end{cases}
\end{split}
\]
It is clear from the definition that the functions
$v_B$ and $v_S$ take into account the big and the small
jumps of $v$ respectively, and that the function
$v_C := v - v_B - v_S$
is continuous in $]a,b[$.
In addition, $v_C$ is uniformly continuous in $]a,b[$,
since it can be continuously extended to $[a,b]$.
Then there exists $\delta > 0$ such that
\begin{equation}\label{f:fc}
|v_C(x) - v_C(y)| < \frac{\eps}{3}
\qquad \forall x,y\in ]a,b[,\ |x-y| < \delta.
\end{equation}
Moreover, from (\ref{f:resto}) we have that
\begin{equation}\label{f:fs}
|v_S(x)| \leq \sum_{j>N} |v(x^j_+) - v(x^j_-)| < \frac{\eps}{3}\,,
\quad\forall x\in ]a,b[\, .
\end{equation}

Let $J_{\eps} = \{y^j\}_{j=0}^m$,
with $a=y^0<y^1<\cdots < y^m = b$, be a partition of $[a,b]$
satisfying the following properties:
\begin{itemize}
\item[(a)]
$y^i - y^{i-1} < \delta$ for every $i\in \{1,\ldots,m\}$;
\item[(b)]
$x^j\in J_{\eps}$ for every $j\in \{1,\ldots,N\}$;
\item[(c)]
for every $i\in\{1,\ldots, m-1\}$, if $y^i = x^j$ for some $j\in \{1,\ldots,N\}$,
then $y^{i-1}, y^{i+1}\not\in\{x_1,\ldots, x_N\}$;
\item[(d)]
$P\cap J_{\eps}=\emptyset$; moreover,
each interval $]y^{i-1}, y^i[$ contains at most one point of $P_{\eps}$ and,
in that case, $y^{i-1}, y^{i}\not\in\{x_1,\ldots, x_N\}$.
\end{itemize}

Let $i\in \{1,\ldots,m\}$.
For every $x,y\in ]y^{i-1}, y^i[$ we have that
\[
\begin{split}
|v(x) - v(y)| \leq {} &
|v(x) - v_C(x) - v_B(x)| + |v(y) - v_C(y) - v_B(y)|
\\ & +
|v_C(x)-v_C(y)| + |v_B(x)-v_B(y)|\,.
\end{split}
\]
Since $]y^{i-1}, y^i[$ does not contain points of
$\{x^1,\ldots, x^N\}$ we have that
$v_B(x) = v_B(y)$.
Moreover, $v-v_C-v_B = v_S$, hence by (\ref{f:fc}) and (\ref{f:fs})
we obtain
\begin{equation}\label{f:oscill}
|v(x)-v(y)| \leq |v_S(x)| + |v_S(y)| + |v_C(x) - v_C(y)| < \eps
\end{equation}
($x,y\in ]y^{i-1}, y^i[$, $i = 1,\ldots, m$).

Finally, let us define the function $v_{\eps}\colon ]a,b[\to\R$ by
$v_{\eps}(y^i) = v(y^i)$ for every $i\in \{0,\ldots,m\}$, and,
on every interval $]y^{i-1}, y^i[$ ($i\in \{1,\ldots,m\}$) by
\[
v_{\eps}(x) :=
\begin{cases}
v(\overline{x}),
&\text{if $\emptyset\neq P_{\eps}\cap ]y^{i-1}, y^i[ = \{\overline{x}\}$},\\
v(y^i_-),
&\text{if $y^i\in \{x^1,\ldots,x^N\}$},\\
v(y^{i-1}_+),
&\text{otherwise}\,.
\end{cases}
\]
It is clear from the construction that (i)--(iv) hold.
Moreover, on every interval $]y^{i-1}, y^i[$ ($i\in \{1,\ldots,m\}$)
we have that $v_{\eps}(y^i_-) = v(y^i_-)$ or
$v_{\eps}(y^i_+) = v(y^i_+)$
or $v_{\eps}(\overline{x}) = v(\overline{x})$ for some
$\overline{x}\in ]y^{i-1}, y^i[$,
hence from (\ref{f:oscill}) we conclude that also (v) holds.
\end{proof}

\begin{lemma}\label{l:approx}
Let $\u\in BV(]a,b[;\R^d)$ be a good representative, let
$J$ denote its jump set, and let $P\subset ]a,b[\setminus J$ be
a countable set.
Then there exists a sequence of piecewise constant functions
$\u_n\in BV(]a,b[, \R^d)$, $n\in\bbbn$,
satisfying the following properties:
\begin{itemize}
\item[(i)]
the (finite) jump set $J_n$ of $\u_{n}$ does not contain points of $P$
and contains all jumps
$x\in J$ such that $|\u(x_+)-\u(x_-)| > 1/n$;
\item[(ii)]
$TV(\u_{n}) \leq TV(\u)$;
\item[(iii)]
$|\u_{n}(x) - \u(x)| < C/n$ for every $x\in ]a,b[$ and $n\in\bbbn$,
where $C = 3\sqrt{d}$;
\item[(iv)]
for every $x\in P$ there exists $n_x\in\bbbn$ such that
$\u_{n}(x) = \u(x)$
for every $n\geq n_x$;
\item[(v)]
for every $x\in J$ there exists $n_x\in\bbbn$ such that
$\u_{n}(x_{+}) = \u(x_{+})$, $\u_{n}(x_{-}) = \u(x_{-})$
for every $n\geq n_x$.
\end{itemize}
\end{lemma}

\begin{proof}
Let $P = \{z_j\}_j$.
For every $n\in\bbbn$ let us apply Lemma~\ref{l:approx1} to each
component $u^i$, $i=1,\ldots,d$ with $\eps = 3/n$
and $P_{\eps} = \{z_1, \ldots, z_n\}$.
The conclusion follows from the fact that
$J = \bigcup_{i=1}^d J_{u^i}$ and
$J_n = \bigcup_{i=1}^d J_{u^i_n}$, $n\in\bbbn$.
\end{proof}

\section{A chain rule formula in $BV(]a,b[;\R^{d})$}

Let $B:]a,b[\times\R^{d}\rightarrow\R$ be a function such that $B\left(\cdot, \w\right)  \in BV(]a,b[)$ for all $\w\in\R^{d}$\,.
We recall that for every $\w\in \R^d$
\begin{equation}\label{dec}
(D_x B)(\cdot,\w)=(\nabla_x B)(\cdot,\w)\,dx+(D^{c}_x B)(\cdot,\w)+
\sum_{x\in \mathcal N_{\w}}
\left[B(x_{+},\w)-B(x_{-},\w)\right]\delta_x
\end{equation}
is the usual decomposition of the measure $(D_x B)(\cdot,\w)$ with respect to the Lebesgue measure,
where $\mathcal N_{\w}:=J_{B(\cdot,\w)}$ is the jump set of $B(\cdot,\w)$\,.

\begin{theorem}\label{chain}
Let $B:]a,b[\times\R^{d}\rightarrow\R$ be a locally bounded function such that

\begin{enumerate}
\item[(A1)] for all $\w\in\R^{d}$ the function $B\left(\cdot, \w\right)$
belongs to $BV(]a,b[)$ and there exists a countable set $\mathcal N\subset ]a,b[$ such that for every $\w\in\R^{d}$ we have
$$
\mathcal N_{\w}\subseteq\mathcal N\,;
$$

\item[(A2)] for every compact set $M\subseteq\R^{d}$ there exists a finite positive Borel measure $\mu_{M}$ in  $]a,b[$ such that for every $\w,\w'\in M$ and every Borel set $A\subseteq ]a,b[$
\[
|(D_x B)(\cdot,\w)-(D_x B)(\cdot,\w')|(A)\leq |\w-\w'|\mu_{M}(A)\,;
\]

\item[(A3)]  for all $x\in ]a,b[\setminus \mathcal N$ the function $B\left(x,\cdot\right)$
belongs to $C^{1}(\R^{d})$ and, for every compact set $M\subset\R^d$,
there exists a constant $D_M>0$ such that
\[
|(D_{\w} B)(x, \w)| \leq D_M,\qquad
\forall x\in ]a,b[\setminus \mathcal N\,, \w\in M\,;
\]

\item[(A4)]
the function  $(D_{\w}B)\left(\cdot,\w\right)$
belongs to $BV(]a,b[;\R^d)$ for every $\w\in\R^{d}$;

\item[(A5)]  there exists a positive finite Cantor measure $\lambda$ on $]a,b[$ 
such that $(D^{c}_x B)(\cdot,\w) \ll \lambda$
for every $\w\in \R^d$.
\end{enumerate}
Then for every $\u\in BV(]a,b[;\R^{d})$ the composite function
$v(x):=B\left(x, \u(x)\right)$,  $x\in]a,b[$,
belongs to $BV(]a,b[)$
and for any
$\phi\in C_0^1(]a,b[)$ we have
\begin{equation}\label{f:int}
\begin{split}
\int_{]a,b[} \phi'(x)  v(x)\,dx = {} &
-\int_{]a,b[} \phi(x) (\nabla_x B)(x,\u(x))\, dx
\\ & -
\int_{]a,b[} \phi(x) \psi(x,\u(x))\, d\lambda
\\ & -
\int_{]a,b[} \!\phi(x) (D_{\w}B)(x,\u(x))\cdot\nabla \u(x)\,dx
\\ & -
\int_{]a,b[}\phi(x)({D_{\w}B})(x,\u(x))\,\cdot dD^c\u(x)
\\ & -
\sum_{x\in  \mathcal N\cup J_\u}\!\!\phi(x)
\left[B(x_{+},\u(x_{+}))-B(x_{-},\u(x_{-}))\right]\,,
\end{split}
\end{equation}
where for every $\w\in \R^d$ the function $\psi(\cdot,\w)$ is the Radon-Nikod\'ym derivative of the measure $(D^{c}_x B)(\cdot,\w) $ with respect to $\lambda$, i.e.
$$
\psi(\cdot,\w):=\frac{d(D^{c}_x B)(\cdot,\w)}{d\lambda}\,.
$$
\end{theorem}

\begin{remark}\label{ipo}
By $(A2)$ we obtain that for every compact set $M\subseteq\R^{d}$ there exists a constant
$C_M$  such that
\begin{equation}\label{dom}
|(D_x B)(\cdot,\w)|(]a,b[)\leq C_{M}\,,\quad \forall \w\in M\,.
\end{equation}
Moreover, for a.e.\ $x\in ]a,b[$ we have that
\begin{equation}\label{dom1}
|(\nabla_x B)(x,\w)|\leq C_{M}\,,\quad \forall \w\in M\,,
\end{equation}
and, for every $x\in ]a,b[$,
\begin{equation}\label{dom12}
|\psi(x,\w)|\leq C_{M}\,,\quad \forall \w\in M\,.
\end{equation}
In addition, for a.e.\ $x\in ]a,b[$ and for every Borel set $A\subseteq ]a,b[$
the functions $\w\mapsto (\nabla_x B)(x,\w)$ and $\w\mapsto (D^{c}_x B)(\cdot,\w)(A)$
are Lipschitz continuous in $\R^d$.
Finally, for $\lambda$-a.e.\  $x\in ]a,b[$ the function 
$\w\mapsto \psi(x,\w)$ is Lipschitz continuous too.
\end{remark}

\begin{remark}
Formula $(\ref{f:int})$ can be rewritten in a more explicit way as
\begin{equation}\label{f:int1}
\begin{split}
\int_{]a,b[} & \phi'(x)  v(x)\,dx = {} 
\\ & -\int_{]a,b[} \phi(x) (\nabla_x B)(x,\u(x))\, dx
-
\int_{]a,b[} \!\phi(x) (D_{\w}B)(x,\u(x))\cdot\nabla \u(x)\,dx
\\ & -
\int_{]a,b[} \phi(x) \psi(x,\u(x))\, d\lambda
 -
\int_{]a,b[}\phi(x)({D_{\w}B})(x,\u(x))\,\cdot dD^c\u(x)
\\ & -
\sum_{x\in\mathcal N}\!\!\phi(x)
\left[\frac{B(x_{+},\u(x_{+}))+B(x_{+},\u(x_{-}))}{2}
-\frac{B(x_{-},\u(x_{+}))+B(x_{-},\u(x_{-}))}{2}\right]
\\ & -
\sum_{x\in J_\u}\!\!\phi(x)
\left[B^{*}(x,\u(x_{+}))-B^{*}(x,\u(x_{-}))\right]\,,
\end{split}
\end{equation}
where for every $x\in ]a,b[$ and $ \w\in \R^{d}$
$$
B^{*}(x,\w):=\frac{B(x_{+},\w)+B(x_{-},\w)}{2}
$$
is the precise representative
of the $BV$ function $x\mapsto B(x,\w)$\,.

In fact, it is easy to check that for every $x\in J_\u\cap \mathcal N$ we have
\begin{align*}
& B(x_+,\u(x_{+}))-B(x_-,\u(x_{-}))
 = \left[B^{*}(x,\u(x_{+}))-B^{*}(x,\u(x_{-}))\right]
\\ & +
\frac{B(x_{+},\u(x_{+}))+B(x_{+},\u(x_{-}))}{2}
-\frac{B(x_{-},\u(x_{+}))+B(x_{-},\u(x_{-}))}{2}\,;
\end{align*}
in particular, for every $x\in J_\u\setminus \mathcal N$ we have
\begin{align*}
B(x_+,\u(x_{+}))-B(x_-,\u(x_{-}))
 & = B(x,\u(x_{+}))-B(x,\u(x_{-}))
 \\ & =
B^{*}(x,\u(x_{+}))-B^{*}(x,\u(x_{-}))\,,
\end{align*}
and
for every $x\in \mathcal N\setminus J_\u$ we have
\begin{align*}
& B(x_+,\u(x_{+}))-B(x_-,\u(x_{-}))
 = B(x_+,\u(x))-B(x_-,\u(x))
 \\ & =
 \frac{B(x_{+},\u(x_{+}))+B(x_{+},\u(x_{-}))}{2}
-\frac{B(x_{-},\u(x_{+}))+B(x_{-},\u(x_{-}))}{2}\,.
\end{align*}
\end{remark}

\bigskip\noindent
{\bf Proof of Theorem \ref{chain}}\par

Since the proof of Theorem~\ref{chain} is rather long,
it will be convenient to divide it into several steps.

In Step~1, following the regularization argument of
Ambrosio--Dal Maso (see \cite{adm}),
we consider the mollification $B_{\varepsilon}(x, \w)$ of
$B(x, \w)$ with respect to the first variable.
We observe that, for every test function $\phi\in C^1_{c}(]a,b[)$,
the integral
\[
\int_{]a,b[} \phi'(x) B_{\varepsilon}(x, \u(x))\, dx
\] 
converges to the left-hand side of (\ref{f:intro})
as $\varepsilon\to 0^+$ 
(see (\ref{dim555})).
Then, for $\varepsilon$ small enough, 
we decompose this
integral (using the chain rule formula for $C^1$ functions)
as 
\[
\begin{split}
-\int_{]a,b[} \phi'(x) B_{\varepsilon}(x, \u(x))\, dx
& = \int_{]a,b[} \phi(x)\, D_{\w} B_{\varepsilon}(x, \u(x))\, d \tilde{D}\u(x)
\\ & +
\sum_{x\in J_{\u}} \phi(x)\, \left[
B_{\varepsilon}(x, \u(x_+)) - B_{\varepsilon}(x ,\u(x_-))\right]
\\ & +
\int_{]a,b[} \phi(x)\, (D_x B_{\varepsilon}) (x, \u(x))\, dx
=: D_{\varepsilon} + J_{\varepsilon} + I_{\varepsilon}\,, 
\end{split}
\]
and we study the convergence of each one of the
three terms 
$D_{\varepsilon}$, $J_{\varepsilon}$, $I_{\varepsilon}$
appearing at the right-hand side as $\varepsilon\to 0^+$.

The limits of $D_{\varepsilon}$ and $J_{\varepsilon}$
are computed respectively in Steps~2 and~3
following the lines of \cite{dcfv2}.

The limit of $I_{\varepsilon}$ is far more difficult to analyze,
because of the possible interaction between the jump set of 
$\u$ and the jump set of $B(\cdot,\u)$.
In Step~4 we compute this limit in the special case of
$\u$ piecewise constant.
Finally, the general case is proved in Step~5
relying on a carefully chosen approximation of a BV function by means of
piecewise constant functions, 
whose construction has been shown in Lemma~\ref{l:approx}.

\medskip\noindent
{\sc Step 1.} \ Fix $\phi\in C_{c}^{1}(]a,b[)$ and let
$\varphi_{\varepsilon}=\varphi_{\varepsilon}\left(  x\right)  $ be a standard
family of mollifiers.  
Let us define
\[
B_{\varepsilon }\left(  x, \bf{w}\right)  :=
\int_{]a,b[}\varphi_{\varepsilon}\left(  x-y\right)  \,B\left(  y, \w\right)\,dy
\]
for $x\in]a',b'[$ and $ \w\in\R^{d},$ where
$  {supp}\  \phi\subset[a',b']\subset]a,b[\,,$
and $0<\varepsilon<  min\{b-b',a-a'\}$\,.

We claim that $B_{\varepsilon}\in C^{1}(]a',b'[\times \R^{d})$.
Firstly we prove that $D_x B_{\varepsilon}$ is locally Lipschitz continuous
in $]a',b'[\times \R^{d}$. In fact, by hypothesis (A3) for every compact set
$D\subseteq ]a',b'[\times \R^{d}$ and
for every $(x_1,\w_1), (x_2,\w_2)\in D$  there exists a constant $C_D$ such that
\[
\begin{split}
|(D_xB_\eps) (x_1,\w_1) & -(D_xB_\eps) (x_2,\w_2)|
\\ = {} &
\Big|\int_{]a,b[}\big[\varphi'_{\varepsilon}(x_1-y)B(y,\w_1)-\varphi'_{\varepsilon}(x_2-y)B(y,\w_2)\big]\,dy\Big|
\\ \leq {} &
\Big|\int_{]a,b[}\big[\varphi'_{\varepsilon}(x_1-y)-\varphi'_{\varepsilon}(x_2-y)\big]B (y,\w_1)\,dy\Big|
\\ & +
\Big|\int_{ ]a,b[}\varphi'_{\varepsilon}(x_2-y)\big[B (y,\w_1)-B(y,\w_2)\big]\,dy\Big|
\\ \leq {} & 
\frac{C_D}{\eps} \big(|x_1-x_2|+|\w_1-\w_2|\big)\,.
\end{split}
\]
Moreover, we prove that
$D_{\w} B_{\varepsilon}$ is continuous in $]a',b'[\times \R^{d}$.
In fact, for every sequence $(x_n,\w_n)$ converging to $(x,\w)$ in $]a',b'[\times \R^{d}$ we have
\[
\begin{split}
|(D_\w B_\eps) (x_n,\w_n) & -(D_\w B_\eps) (x,\w)|
\\ = {} &
\Big|\int_{]a,b[}\big[\varphi_{\varepsilon}(x_n-y)(D_\w B)(y,\w_n)-\varphi_{\varepsilon}(x-y)(D_\w B)(y,\w)\big]\,dy\Big|
\\ \leq {} &
\int_{ ]a,b[}\big|\varphi_{\varepsilon}(x_n-y)-\varphi_{\varepsilon}(x-y)\big|\,\big|(D_\w B)(y,\w_n)\big|\,dy
\\ & +
\int_{]a,b[}\big|\varphi_{\varepsilon}(x-y)\big|\,\big|(D_\w B)(y,\w_n)-(D_\w B)(y,\w)\big|\,dy\,.
\end{split}
\]
The first integral tends to 0, as $n\to\infty$, since by (A3)
\[
\int_{ ]a,b[}\big|\varphi_{\varepsilon}(x_n-y)-\varphi_{\varepsilon}(x-y)\big|\,\big|(D_\w B)(y,\w_n)\big|\,dy
\leq \frac{C}{\eps} D_M |x_n-x|\,,
\]
and the second one  tends to 0, as $n\to\infty$, by the continuity of the function $(D_{\w} B)(y,\cdot)$ for a.e. $y\in]a,b[$, the boundedness of $D_{\w} B$ and
by the Lebesgue dominated convergence theorem.

Let $\u\in BV(]a,b[;\R^{d})$ and define
\[
v_{\varepsilon}\left(  x\right)  :=B_{\varepsilon}(x,\u(x)),\quad\quad x\in ]a',b'[\,.
\]
Since $B_{\varepsilon}\in C^{1}(]a',b'[\times \R^{d})$
we can apply the chain rule formula (see Theorem 3.96 in \cite{AFP}) to the composition of the function
$B_{\varepsilon}$ with the $BV$ map $x\mapsto (x,\u(x))$,
concluding that $v_{\varepsilon}\in BV(]a',b'[)$ and
\begin{equation}\label{chainrule-eps}
\begin{split}
\int_{]a',b'[} & \phi'(x) v_{\varepsilon}(x)\ dx
= 
- \int_{]a',b'[} \phi(x)\,(D_x B_{\varepsilon})\big(x,\u(x))\ dx
\\ & -
\int_{]a',b'[}\phi(x)(D_{\u}B_{\varepsilon})(x,\u(x))\,\cdot d\widetilde D\u(x)
\\ & -
\sum_{x\in J_\u\cap ]a',b'[}\!\!\phi(x)
\big(B_{\varepsilon}(x,\u(x_{+}))-B_{\varepsilon}(x,\u(x_{-}))\big)
\\ = {} &
- \int_{]a',b'[} \phi(x)\,(D_x B_{\varepsilon})\big(x,\u(x))\ dx
\\ & -
\sum_{i=1}^{d}\int_{]a',b'[}\phi(x)(D_{w^{i}}B_{\varepsilon})(x,
\u(x)) d\widetilde Du^{i}(x)
\\ & -
\sum_{i=1}^{d}\sum_{x\in J_{u^{i}}\cap ]a',b'[}\!\!\phi(x)\big(u^{i}(x_{+})-u^{i}(x_{-})\big)
\int_{0}^{1}(D_{w^{i}}B_{\varepsilon})\big(x,\w_s(x))\big)\ ds\,,
\end{split}
\end{equation}
where $\widetilde D\u$ and $\widetilde Du^{i}$ denote the diffuse parts of the measures ${D\u}$
and ${Du^{i}}$ respectively,
and $\w_s(x):=\u(x_{-})+s(\u(x_{+})-\u(x_{-}))$.

Since $B$ is locally bounded and the functions $B_\eps(\cdot,\w)$ converge a.e. in
$]a',b'[$ to $B(\cdot,\w)$, by Lebesgue dominated convergence theorem we get
\begin{equation}\label{dim555}
\lim_{\varepsilon\rightarrow0^{+}}\int_{]a,b[}\phi'(x)
B_{\varepsilon }\left(x,\u(x)\right)  dx=\int_{]a,b[}\phi'(x)
B_{ }\left(x,\u(x)\right)  dx\,.
\end{equation}

\medskip\noindent
{\sc Step 2.} \ We shall prove  the convergence of the diffuse part, i.e. for every $i=1,\dots,d$ we prove that
\begin{equation}\label{dim55}
\lim_{\varepsilon\rightarrow0^{+}}
\int_{]a',b'[}\phi(x)(D_{w^{i}}B_{\varepsilon})(x,\u(x)) d\widetilde Du^{i}=
\int_{]a,b[}\phi(x)(D_{w^{i}}B)(x,\u(x)) d\widetilde Du^{i}\,.
\end{equation}
Using the coarea formula (\ref{coarea}), we get
\begin{eqnarray}\label{perpar7}
\qquad
&&\int_{]a',b'[}\phi(x)
(D_{w^{i}}B_{\varepsilon})(x, u^{1}(x),\dots, u^{i}(x),\dots, u^{d}(x))\, d\widetilde Du^{i}\\
&=&
\int_{ ]a',b'[\cap C_{u^{i}}}\phi(x)(D_{w^{i}}B_{\varepsilon})(x, u^{1}(x),\dots,u^{i}(x),\dots,u^{d}(x)) \frac{\widetilde Du^{i}}{{|Du^{i}|}}(x)\,d|Du^{i}|\nonumber \\
&=&
\int_{-\infty}^{+\infty}\!dt\!\int_{\{u^i_{-}\leq t\leq u^i_{+}\}\cap C_{u^{i}}}\!\phi(x)
(D_{w^{i}}B_\eps)(x,{\u}(x))\frac{\widetilde Du^{i}}{{|Du^{i}|}}(x)\,d\ho \nonumber \\
&=&
\int_{-\infty}^{+\infty}\!dt\!\int_{\{{u^{i}}=t\}\cap C_{u^{i}}}
\phi(x)(D_{w^{i}}B_\eps)(x,u^{1}(x),\dots,t,\dots,u^{d}(x))
\frac{\widetilde Du^{i}}{{|Du^{i}|}}(x)\, d\ho\,. \nonumber
\end{eqnarray}
Now, by (A4) we have that
for every $i=1,\dots,d$  and for every $\w\in\R^{d}$
\begin{equation}\label{dim74}
(D_{w_{i}}B_\eps)(x,\w)\to (D_{w_{i}}B)^*(x,\w)\qquad \forall x\in]a,b[
\end{equation}
as $\eps\to0$.
Therefore,
for a.e. $t\in\R$, we have
\[
\begin{split}
\lim_{\eps\to0} & \int_{\{{u^{i}}=t\}\cap C_{u^{i}}}
\phi(x)(D_{w^{i}}B_\eps)(x,u^{1},\dots,t,\dots,u^{d})
\frac{\widetilde Du^{i}}{{|Du^{i}|}}\, d\ho
\\ & =
\int_{\{{u^{i}}=t\}\cap C_{u^{i}}}
\phi(x)(D_{w^{i}}B)^{*}(x,u^{1},\dots,t,\dots,u^{d})
\frac{\widetilde Du^{i}}{{|Du^{i}|}}\, d\ho\,.
\end{split}
\]
{}From this equation, using the local boundedness of $(D_{w_{i}}B)^*$ and the fact
that, by the coarea formula (\ref{coarea}),
$$
\int_{-\infty}^{+\infty}
\ho\left(\{{u^{i}}=t\}\cap C_{u^i}\right)dt=
|Du^i|(C_{u^{i}})<\infty\,,
$$
we can pass to the limit in (\ref{perpar7}) and by Lebesgue
dominated convergence theorem we get
\[
\begin{split}
\lim_{\eps\to 0} &
\int_{]a',b'[}\phi(x)
(D_{w^{i}}B_{\varepsilon})(x,u^{1}(x),\dots,u^{i}(x),\dots,u^{d}(x)) d\widetilde Du^{i}
\\ & =
\int_{-\infty}^{+\infty}\!dt
\int_{\{{u^{i}}=t\}\cap C_{u^{i}}}
\phi(x)(D_{u^{i}}B)^{*}(x,u^{1},\dots,t,\dots,u^{d})
\, d \widetilde Du^{i}\,.
\end{split}
\]
{}From this equation, using the coarea formula (\ref{coarea}) again, we
immediately get (\ref{dim55}).

\medskip\noindent
{\sc Step 3.} \ We shall prove  the convergence of the jump part, i.e. for every $i=1,\dots,d$ we prove that
\begin{equation}\label{dim5}
\begin{split}
& \lim_{\varepsilon\rightarrow0^{+}}
\sum_{x\in J_{u^{i}}\cap ]a',b'[}\!\!\phi(x)\big(u^{i}(x_{+})-u^{i}(x_{-})\big)
\int_{0}^{1}(D_{w^{i}}B_{\varepsilon})\big(x,\w_s(x))\big)\, ds
\\ & =
\sum_{x\in J_{u^{i}}\cap ]a',b'[}\!\!\phi(x)\big(u^{i}(x_{+})-u^{i}(x_{-})\big)
\int_{0}^{1}(D_{w^{i}}B)^*\big(x,\w_s(x))\big)\, ds\,,
\end{split}
\end{equation}
where $\w_s(x):=\u(x_{-})+s(\u(x_{+})-\u(x_{-}))$.
Let us fix $i\in\{1,\dots,d\}$, 
and let $J_{u^{i}}' := J_{u^{i}}\cap ]a',b'[ = \{y^j\}_{j\in\bbbn}$. 
For every $h\in\bbbn$ there exists
$k(h)\in\bbbn$ such that
\[
\sum_{j=k(h)+1}^{\infty}|u^{i}(y^j_{+})-u^{i}(y^j_-)|<
\frac{1}{h}.
\]
Then the following estimate holds:
\begin{align*}
& \left|\sum_{x\in J_{u^{i}}'}\phi(x)(u^{i}(x_{+})-u^{i}(x_{-}))
\int_{0}^{1}\Big((D_{w^{i}}B_{\varepsilon})(x,\w_s(x))-
(D_{w^{i}}B)^*(x,\w_s(x))\Big)ds\right|
\\ & \leq
\|\phi\|_\infty\sum_{j=1}^{k(h)}\big| u^{i}(y^j_{+})-u^{i}(y^j_{-})\big|\!
\int_{0}^{1}\Big|(D_{w^{i}}B_{\varepsilon})(y^j,\w_s(y^j))-
(D_{w^{i}}B)^*(y^j,\w_s(y^j)) \Big|ds
\\ & +
\|\phi\|_\infty\sum_{j>k(h)}\big| u^{i}(y^j_{+})-u^{i}(y^j_{-})\big|\!
\int_{0}^{1}\Big|(D_{w^{i}}B_{\eps})(y^j,\w_s(y^j))-
(D_{w^{i}}B)^*(y^j,\w_s(y^j)) \Big| ds
\\ & \leq
\|\phi\|_\infty
\int_{0}^{1}\sum_{j=1}^{k(h)}\big| u^{i}(y^j_{+})-u^{i}(y^j_{-})\big|\!\Big|(D_{w^{i}}B_{\varepsilon})(y^j,\w_s(y^j))-
(D_{w^{i}}B)^*(y^j,\w_s(y^j)) \Big|ds
\\ & +
2C\|\phi\|_\infty\sum_{j>k(h)}\big| u^{i}(y^j_{+})-u^{i}(y^j_{-})\big|\, ,
\end{align*}
where $C :=\|D_{w^{i}}B\|_{L_\infty(]a',b'[\times ]-M,M[)}$ and $\|u^i\|_\infty\leq M$.
By Lebesgue dominated convergence theorem, the first integral is infinitesimal as $\eps\to0$, since $(D_{w^{i}}B_\eps)(x,\w)$ and $(D_{w^{i}}B)^*(x,\w)$ are locally bounded functions and for every $x\in]a',b'[$ and
$\w\in\R^d$ we have that $(D_{w^{i}}B_\eps)(x,\w)\to (D_{w^{i}}B)^*(x,\w)$, as $\eps\to 0$.
Therefore, letting
first $\eps$ tend to zero and then $h$ tend to $\infty$,
we immediately obtain (\ref{dim5}).

\medskip\noindent
{\sc Step 4.}
In this step, we consider a piecewise constant function  $\u\colon ]a,b[\to\R^d$ of the form
\[
\u(x) = \sum_{i=0}^N \v^i \Xi(x),
\]
where $\v^0,\ldots, \v^N\in\R^d$,
$a = a_0 < a_1 < \ldots < a_N < a_{N+1} = b$ and we prove that
\begin{equation}\label{nnnn}
\begin{split}
{} &  \lim_{\varepsilon\to 0}\int_{]a',b'[} \phi(x)\, (D_{x}B_{\varepsilon})(x,\u(x))\, dx
\\ & =
\int_{]a',b'[}  \phi(x)\, (\nabla_x B)(x,\u(x))\,dx
+ \int_{]a',b'[}  \phi(x)\,  
\frac{d\,(D^{c}_x B)(\cdot,\u)}{d\,\lambda}d\,\lambda
\\ & +
\sum_{x\in \mathcal N\cup J_{\u}}
\phi(x)\left[
\frac{B(x_+, \u(x_+)) + B(x_+, \u(x_-))}{2} - \frac{B(x_-, \u(x_+) + B(x_-, \u(x_-))}{2}
\right]
\,.
\end{split}
\end{equation}
In order to simplify the notation, let us denote by $\chi_i$
the characteristic function $\chi_{[a_i, a_{i+1}]}$.
{}From the very definition of $\u$
and (\ref{f:Leib}), we have that
\[
\begin{split}
I_{\eps} := {} &
\int_{]a',b'[} \phi(x)\, (D_{x}B_{\varepsilon})(x,\u(x))\, dx
\\ = {} &
\sum_{i=0}^N \int_{]a',b'[} \phi(x)\, \chi_i^*(x)\, (D_{x}B_{\varepsilon})(x,\v^i)\, dx
\\ = {} &
- \sum_{i=0}^N \int_{]a',b'[} B_{\varepsilon}(x,\v^i)\, dD(\phi\chi_i)
\\ = {} &
- \sum_{i=0}^N \int_{]a',b'[} \phi'(x)\, \chi_i(x)\, B_{\varepsilon}(x,\v^i)\, dx
\\ & + \sum_{i=0}^N \left[\phi(a_{i+1})\, B_{\varepsilon}(a_{i+1},\v^i)
- \phi(a_{i})\, B_{\varepsilon}(a_{i},\v^i)\right]\,.
\end{split}
\]
Passing to the limit as $\eps \to 0$ we obtain
\begin{equation}\label{f:eqlim}
\begin{split}
\lim_{\eps\to 0} I_{\eps} = {} &
- \sum_{i=0}^N \int_{]a,b[} \phi'(x)\, \chi_i(x)\, B(x,\v^i)\, dx
\\ & +
\sum_{i=0}^N \left[\phi(a_{i+1})\, B^*(a_{i+1},\v^i)
- \phi(a_{i})\, B^*(a_{i},\v^i)\right]\,.
\end{split}
\end{equation}
Let us consider the integrals at the right-hand side of (\ref{f:eqlim}).
Using again (\ref{f:Leib}) we have that
\[
\begin{split}
-\int_{]a,b[} & \phi'(x)\, \chi_i(x)\, B(x,\v^i)\, dx
= 
\int_{]a,b[} \phi(x)\, dD(\chi_i\, B(\cdot,\v^i))
\\ = {} &
\int_{]a,b[} \phi(x)\, \chi_i^*(x)\, d (D_xB)(\cdot,\v^i)
+ \int_{]a,b[} \phi(x)\, B^*(x,\v^i)\, dD\chi_i
\\ = {} &
\int_{]a,b[} \phi(x)\, \chi_i^*(x)\, d (D_xB)(\cdot,\v^i)
+
\phi(a_{i})\, B^*(a_{i},\v^i) - \phi(a_{i+1})\, B^*(a_{i+1},\v^i).
\end{split}
\]
Substituting this expression into (\ref{f:eqlim})
we thus obtain
\[
I := \lim_{\eps\to 0} I_{\eps} =
\sum_{i=0}^N\int_{]a,b[} \phi(x)\, \chi_i^*(x)\, d (D_xB)(\cdot,\v^i)\,.
\]
Finally, let us decompose each measure $(D_x B)(\cdot, \v^i)$ in the
canonical way (\ref{dec}).
It is not difficult to check that
\begin{equation}\label{asda}
\begin{split}
I
= {} &
\sum_{i=0}^N\int_{]a,b[} \phi(x)\, \chi_i^*(x)\, \nabla_x B(x,\v^i)\, dx
+ \sum_{i=0}^N\int_{]a,b[} \phi(x)\, \chi_i^*(x)\, 
\frac{d D_x^c B(\cdot,\v^i)}{d\,\lambda}\,d\,\lambda
\\ & +
\sum_{i=1}^N \phi(a_i)\left[
\frac{B(a_i+, \v^i) + B(a_i+, \v^{i-1})}{2} - \frac{B(a_i-, \v^i) + B(a_i-, \v^{i-1})}{2}
\right]
\\ & +
\sum_{x\in\mathcal{N}\setminus J_{\u}}
\phi(x) \left[B(x_+, \u(x)) - B(x_-, \u(x))\right]\,.
\end{split}
\end{equation}
The first two terms coincide respectively with
\[
\int_{]a,b[} \phi(x)\, \nabla_x B(x,\u(x))\, dx\,, \qquad
\int_{]a,b[} \phi(x)\, \psi(x,\u(x))\, d\lambda\,.
\]
The last two summations take into account the jump points
$x\in J_{\u}$ and $x\in\mathcal{N}\setminus J_{\u}$, respectively.
Again, it is not difficult to check that, 
in both cases,
the corresponding term can always be written as
\[
\phi(x)\left[
\frac{B(x_+, \u(x_+)) + B(x_+, \u(x_-))}{2} - \frac{B(x_-, \u(x_+) + B(x_-, \u(x_-))}{2}
\right]\,,
\]
so that (\ref{nnnn}) follows.

\medskip\noindent
{\sc Step 5.}
In this step, we shall prove  that formula \eqref{nnnn} holds 
for every function $\u\in BV(]a,b[;\R^{d})$, i.e.\ we prove that
\begin{equation}\label{dim66}
I= \lim_{\varepsilon\rightarrow0^{+}}I_\eps\,,
\end{equation}
where
\begin{equation}\label{dim661}
I_\eps:=\int_{]a',b'[}\phi(x)\,(D_x B_{\varepsilon})\big(x,\u(x))\ dx
\end{equation}
and
\[
\begin{split}
I&:= \int_{]a,b[} \phi(x) (\nabla_x B)(x,\u(x))\, dx+
\int_{]a,b[} \phi(x) \psi(x,\u(x))\, d\lambda
\\ & +
\sum_{x\in \mathcal{N}\cup J_{\u}}
\phi(x)\left[
\frac{B(x_+, \u(x_+)) + B(x_+, \u(x_-))}{2} - \frac{B(x_-, \u(x_+) + B(x_-, \u(x_-))}{2}
\right]\,.
\end{split}
\]
Let $\u\in BV(]a,b[;\R^{d})$
and let $(\u_n)_n$ be the sequence of approximating piecewise constant
functions given by Lemma~\ref{l:approx} with $P = \mathcal{N}\setminus J_{\u}$\,.

Fixed $\varepsilon>0$, we set
\begin{equation}\label{dim6611}
I^n_\eps:=\int_{]a',b'[}\phi(x)\,(D_x B_{\varepsilon})\big(x,\u_n(x))\ dx\,.
\end{equation}
By Lebesgue dominated convergence theorem and the continuity of $(D_x B_{\varepsilon})\big(x,\cdot)$ (which follows by $B_{\varepsilon}\in C^{1}(]a',b'[\times \R^{d})$), for every $\varepsilon>0$ we have that

\begin{equation}\label{dim6n}
I_\eps= \lim_{n\to\infty}I^n_\eps\,.
\end{equation}

More precisely, we claim that
\begin{equation}\label{dim53}
|I^n_\eps-I_\eps|\leq\|\u_n-\u\|_\infty \mu_M(]a,b[)\|\phi\|_\infty\quad\forall\varepsilon>0
\quad {\rm small\ enough}.
\end{equation}
Namely, by hypothesis (A2) we have
\begin{equation}\label{dim533}
\begin{split}
|I^n_\eps-I_\eps|\leq &
\int_{]a,b[}|\phi(x)|\Big[\int_{]a,b[}\varphi_{\varepsilon}(x-y)d\big|(D_xB)(\cdot,\u_n(x))-(D_xB)(\cdot,\u(x))\big|(y)\Big]dx
\\ \leq &
\int_{]a,b[}|\phi(x)||\u_n(x)-\u(x)|\Big[\int_{]a,b[}\varphi_{\varepsilon}(x-y)d\mu_M(y)\Big]dx
\\ \leq &
\|\u_n-\u\|_\infty\|\phi\|_\infty
\int_{]a,b[}\Big[\int_{]a,b[}\varphi_{\varepsilon}(x-y)\,dx\Big]\,d\mu_M(y)
\\ = &
\|\u_n-\u\|_\infty\|\phi\|_\infty\mu_M(]a,b[)\,.
\end{split}
\end{equation}

On the other hand, by Step 4 we have for every $n\in\bbbn$
\begin{equation}\label{dim66666}
\lim_{\varepsilon\rightarrow0^{+}}
I^n_\eps = I^n
\end{equation}
where
\begin{equation}\label{dim666667}
\begin{split}
I^n:=
\int_{]a,b[} \phi(x) (\nabla_x B)(x,\u_{n}(x))\, dx
+
\int_{]a,b[} \phi(x) \psi(x,\u_{n}(x))\, d\lambda+ S_n\,,
\end{split}
\end{equation}
and
\[
\begin{split}
S_n := \sum_{x\in \mathcal{N} \cup J_{\u_n}}
\phi(x)\Big[ &
\frac{B(x_+, \u_n(x_+)) + B(x_+, \u_n(x_-))}{2} +
\\ &  - \frac{B(x_-, \u_n(x_+) + B(x_-, \u_n(x_-))}{2}
\Big]\,.
\end{split}
\]

We claim that
\begin{equation}\label{dim60}
I= \lim_{n\to\infty}I^n\,.
\end{equation}
By Remark \ref{ipo} we have that
$$
\lim_{n\to\infty}\int_{]a,b[} \phi(x) (\nabla_x B)(x,\u_{n}(x))\, dx=\int_{]a,b[} \phi(x) (\nabla_x B)(x,\u(x))\, dx
$$
and
$$
\lim_{n\to\infty}\int_{]a,b[} \phi(x) \psi(x,\u_{n}(x))\, d\lambda=\int_{]a,b[} \phi(x) \psi(x,\u(x))\, d\lambda\,.
$$
It remains to show that $S_n$
converges, as $n\to +\infty$, to
\[
S := \sum_{x\in \mathcal{N} \cup J_{\u}}
\phi(x)\left[
\frac{B(x_+, \u(x_+)) + B(x_+, \u(x_-))}{2} - \frac{B(x_-, \u(x_+) + B(x_-, \u(x_-))}{2}
\right]\,.
\]
Let $I_1 = \bigcup_n J_{\u_n}$, $I_2 = \mathcal{N}\setminus I_1$ and $I = I_1 \cup I_2$\,.
We recall that $P = \mathcal{N}\setminus J_{\u}$ and $J_\u\subset I_1$\,.
Since, by construction, $J_{\u_n}\cap P = \emptyset$ for every $n\in\bbbn$, we have that $I_1 \cap P = \emptyset$
and $\mathcal{N}\cup J_{\u}\subset I$.
Hence both summations in $S_n$ and $S$ can be extended to the bigger set $I = \{x_i\}$, since it is easy to check that the
added terms are all zero.
Thus we can write
\[
S_n = \sum_{i\in\bbbn} a^n_i,\quad
S = \sum_{i\in\bbbn} a_i,
\]
where
\[
\begin{split}
a^n_i & :=
\phi(x^i)\left[
\frac{B(x^i_+, \u_n(x^i_+)) + B(x^i_+, \u_n(x^i_-))}{2} - \frac{B(x^i_-, \u_n(x^i_+) + B(x^i_-, \u_n(x^i_-))}{2}
\right],\\
a_i & :=
\phi(x^i)\left[
\frac{B(x^i_+, \u(x^i_+)) + B(x^i_+, \u(x^i_-))}{2} - \frac{B(x^i_-, \u(x^i_+) + B(x^i_-, \u(x^i_-))}{2}
\right].
\end{split}
\]
Let $R \geq \max\{\|\u_n\|_{\infty}, \|\u\|_{\infty}\}$
and let $M = \overline{B}_R(0)$.
{}From assumption (A2) we have that
\[
|B(x_+, \w) - B(x_-, \w)| \leq |B(x_+, 0) - B(x_-, 0)| + R \mu_M(\{x\}),
\qquad x\in ]a,b[,\ \w\in M\,.
\]
Since
\[
|a^n_i| \leq b_i:= \|\phi\|_{\infty}\left[|B(x_+^i,0)-B(x_-^i,0)| + R\mu_M(\{x_i\})\right],
\qquad \forall n\in\bbbn\,,
\]
and
\[
\sum_{i\in\N}b_i \leq \|\phi\|_{\infty} \left[
TV(B(\cdot, 0)) + R\mu_M(]a,b[)\right],
\qquad \forall n\in\bbbn\,,
\]
in order to prove that $\lim_n S_n = S$, by Dominated convergence theorem,
it is enough to prove that $\lim_n a^n_i = a_i$ for every $i\in\bbbn$.

We have three cases.
If $x_i\in P = \mathcal{N}\setminus J_{\u}$, then $\u$ and every $\u_n$ are continuous at $x_i$.
Moreover, for every $n$ large enough, $\u_n(x_i) = \u(x_i)$, hence
$a^n_i = a_i$.
If $x_i\in J_{\u}$, then for every $n$ large enough we have that
$\u_n(x_i+) = \u(x_i+)$, $\u_n(x_i-) = \u(x_i-)$, hence again $a^n_i = a_i$.
Finally, let us consider the case $x_i\in I_1\setminus J_{\u}$.
Since $x_i\not\in\mathcal{N}$, the function $B(x_i, \cdot)$ is continuous in $\R^d$.
Moreover, since $x_i\not\in J_{\u}$, also $\u$ is continuous at $x_i$ and
$\u_n(x_i+), \u_n(x_i-)\to \u(x_i)$, so that $\lim_n a^n_i = a_i$. Therefore \eqref{dim60} is proved.

In order to prove \eqref{dim66}, let us fix $\eta>0$. By \eqref{dim53} and by \eqref{dim60} there exists $n_0\in\N$ such that
$$
|I^{n_0}_\eps-I_\eps|<\frac{\eta}{3}\quad\quad\forall\varepsilon>0
\quad {\rm small\ enough}
$$
and
$$
|I^{n_0}-I|<\frac{\eta}{3}\,.
$$
Moreover by \eqref{dim66666} there exists $\varepsilon_0>0$ such that
$$
|I^{n_0}_\eps-I^{n_0}|<\frac{\eta}{3}\quad\quad\forall0<\varepsilon<\varepsilon_0\,.
$$
Then
$$
|I_\eps-I|\leq |I_\eps-I^{n_0}_\eps|+|I^{n_0}_\eps-I^{n_0}|+|I^{n_0}-I|<\eta
\quad\quad\forall0<\varepsilon<\varepsilon_0\,.
$$
Therefore  \eqref{dim66} is proved and this concludes Step 5.

Finally, the thesis of the theorem is obtained by collecting all the Steps.
\qquad\endproof


\bigskip

In view to the applications to conservation laws (see Proposition \ref{bcbc}) we need to generalize formula (\ref{f:int}) in order to integrate a $BV$ function with respect to the measure $(B(x,\u(x)))_x$\,.

\begin{corollary}\label{chain222}
Let $B:]a,b[\times\R\to\R$ be a function satisfying the same assumptions of Theorem \ref{chain}. Let $g:]a,b[\to \R$ be a $BV$-function such that $J_g\subseteq {\mathcal N}$.

Then for every $ \u\in BV(]a,b[;\R^{d})$ and 
$\phi\in C_0^1(]a,b[)$ we have
\begin{equation}\label{f:int222}
\begin{split}
\int_{]a,b[} \phi(x) g^*(x) & \,d\,(B(x,\u(x)))_x = 
\int_{]a,b[} \phi(x) g(x) (\nabla_x B)(x,\u(x))\, dx
\\ & +
\int_{]a,b[} \phi(x) g(x) \psi(x,\u(x))\, d\lambda
\\ & +
\int_{]a,b[} \!\phi(x) g(x) (D_{\w}B)(x,\u(x))\cdot\nabla \u(x)\,dx
\\ & +
\int_{]a,b[}\phi(x) g(x) ({D_{\w}B})(x,\u(x))\,\cdot dD^c\u(x)
\\ & +
\sum_{x\in  \mathcal N\cup J_\u}\!\!\phi(x) g^*(x)
\left[B(x_{+},\u(x_{+}))-B(x_{-},\u(x_{-}))\right]\,.
\end{split}
\end{equation}
\end{corollary}
\begin{proof} 
Let $g_\eps=g*\varphi_\eps$ be the standard mollified functions of the $BV$-function $g$. We recall that $g_\eps$ pointwise converges (everywhere) in $]a,b[$ to the precise representative $g^*$, as $\eps\to 0$\,. We apply Theorem \ref{chain} by using $\phi(x)g_\eps(x)$ as test function. The conclusion follows by Lebesgue dominated convergence Theorem.
\end{proof}

\bigskip
In the next corollaries we consider  $B(x,\w)$ with a particular structure.

\begin{corollary}\label{chain2}
Let $K\in BV(]a,b[)$ and $f\in C^{1}(\R^{d})$\,.
Then for every $ \u\in BV(]a,b[;\R^{d})$ the function
$v:]a,b[\rightarrow\R$, defined by
\[
v(x):=K(x)f\left(\u(x)\right),\quad x\in]a,b[,
\]
belongs to $BV(]a,b[)$,
and for any
$\phi\in C_0^1(]a,b[)$ we have
\begin{equation}\label{f:int2}
\begin{split}
\int_{]a,b[}\phi'(x)  v(x)\,dx = {} &
-\int_{]a,b[} \phi(x)(f(\u))^{*}(x)\,dD K(x)
\\ & -
\int_{]a,b[}\phi(x)K(x)(\nabla f)(\u(x))\cdot\nabla \u(x)\,dx
\\ & -
\int_{]a,b[}\phi(x)K(x)(\nabla f)(\u(x))\,\cdot dD^c\u(x)
\\ & -
\sum_{x\in J_\u}\!\!\phi(x)
K^*(x)\left[f(\u(x_{+}))-f(\u(x_{-}))\right],
\end{split}
\end{equation}
where $(f(\u))^{*}$ and $K^*$ are the precise representatives
of the $BV$ functions $f(\u)$ and $K$ respectively.
\end{corollary}

\begin{proof}
It is sufficient to observe that the function $B(x,\w):=K(x)f(\w)$ satisfies all the assumptions of Theorem \ref{chain}\,. For instance, hypothesis (A2) is satisfied since for every compact set $M\subset \R^d$ we can choose $\mu_M:=C_M|DK|$, with $C_M=\max_{w\in M}|\nabla f(w)|$, and hypothesis (A5) is satisfied since  we can choose $\lambda:=|DK|$\,.
\end{proof}

\begin{corollary}\label{chain3}
Let $f:\R\times\R^{d}\rightarrow\R$ and $K:]a,b[\rightarrow\R$ be two  functions satisfying

\begin{enumerate}
\item[(i)] $K\in BV(]a,b[)$;

\item[(ii)] the function $f=f(y,\w)$ belongs to $C^{1}(\R\times\R^{d})$\,;

\item[(iii)]
the function  $x\mapsto f_{\w}(K(x), \w)$ belongs to $BV(]a,b[)$ for every $\w\in\R^{d}$\,;

\item[(iv)]
for every compact set $D\subset\R\times\R^d$ there exists a constant $L_D$ such that
$$
|f_{y}(y, \w)-f_{y}(y, \w')|\leq L_D|\w-\w'|\quad\forall (y,\w),(y,\w')\in D
$$
and
$$
|f_{\w}(y, \w)-f_{\w}(y', \w)|\leq L_D|y-y'|\quad\forall (y,\w),(y',\w)\in D\,.
$$

\end{enumerate}
Then for every $ \u\in BV(]a,b[;\R^{d})$ the function
$v:]a,b[\rightarrow\R$, defined by
\[
v(x):=f\left(K(x), \u(x)\right),\quad x\in]a,b[,
\]
belongs to $BV(]a,b[)$,
and for any
$\phi\in C_0^1(]a,b[)$ we have
\begin{equation}\label{f:int3}
\begin{split}
\int_{]a,b[}\phi'(x)  v(x)\,dx = {}
& -
\int_{]a,b[}\phi(x)f_y(K(x),\u(x))\,\cdot d{\widetilde DK}(x)
\\ & -
\int_{]a,b[}\phi(x)f_\w(K(x),\u(x))\,\cdot d\widetilde D\u(x)
\\ & -
\sum_{x\in J_\u\cup J_K}\!\!\phi(x)
\left[f(K(x_+),\u(x_{+}))-f(K(x_-),\u(x_{-}))\right]\,.
\end{split}
\end{equation}
\end{corollary}

\begin{proof}
We observe that the function
 $B(x,\w)=f(K(x),\w)$ satisfies all the assumptions of Theorem \ref{chain}\,. In particular, $\mathcal N=J_K$,
 $$
 D^c_xB(\cdot,\w)=f_y(K(x),\w) D^cK\ll|D^cK|=\lambda
 $$
and
 $\psi(x,\w)=f_y(K(x),\w)\frac{D^cK}{|D^cK|}(x)$.
\end{proof}

\section{Comparison with other chain rule formulas}

In \cite{dcfv2} it was proved a chain rule formula for function $u:\R^N\to \R$. 
In Theorem 5.1 we recall this formula which coincides to formula (\ref{f:int}) in the case of $d=1=N$.
Although the two formulas look like very different, we explicitely show that they concide for piecewise constant functions.

\begin{theorem}\label{chainB} Let $B:]a,b[\times\R\rightarrow\R$ be a locally
bounded Borel function. Assume that $B(x,0)=0$ for all $x\in ]a,b[$ and

\begin{enumerate}
\item[(i)] for all $t\in\R$ the function $B\left(\cdot,t \right)  \in BV(]a,b[)\,;$
\item[(ii)]  for all $x\in ]a,b[$ the function $B\left(x,\cdot\right)$
belongs to $C^{1}(\R)\,;$
\item[(iii)]
the function $D_{t}B$ is locally bounded, 
for all $t\in\R$ the function $(D_{t}B)\left(\cdot,t\right)$ belongs to $BV(]a,b[)$
and for every compact set $M\subset\R$
\[
\int_M|D_x (D_{t} B)(\cdot,t)|(]a,b[)\ dt<+\infty\,.
\]

\end{enumerate}
Then, for every $u\in BV(]a,b[)$
the composite function 
$v(x):=B(x,u(x))$, $x\in ]a,b[$,
belongs to $BV_{loc}(]a,b[)$ and for any
$\phi\in C_0^1(]a,b[)$ we have
\begin{eqnarray}\label{chainrules}
\qquad
\int_{]a,b[}\phi'(x)  v(x)\,dx=
\!\!&-&\!\!\!\!\!\!
\int_{-\infty}^{+\infty}\!\!dt\!\!\int_{]a,b[}{\rm sgn}(t)\chi^*_{\Omega_{u,t}}(x)\phi(x)\,dD_x(D_{t}B)(\cdot,t)\\
\!\!&-&\!\!\!\!\!\!
\!\int_{]a,b[}\!\phi(x)(D_{t}B)(x,u(x))\nabla u(x)\,dx \nonumber\\
\!\!&-&\!\!\!\!\!\!
\int_{]a,b[}\phi(x)(D_{t}B)(x,u(x))\,dD^cu(x)\nonumber\\
\!\!&-&\!\!\!\!\!\!
\sum_{x\in J_u}\!\!\phi(x)
\left[B^{*}(x,u(x_{+}))-B^{*}(x,u(x_{-}))\right]\,,\nonumber
\end{eqnarray}
where $\Omega_{u,t}=\{x\in]a,b[:\,t$ belongs to the segment of endpoints $0$ and $u(x)\}$ and $\chi^*_{\Omega_{u,t}}$ and
$B^{*}(\cdot,t)$ are, respectively, the precise representatives of the $BV$ functions
$\chi_{\Omega_{u,t}}$ and $B(\cdot,t)$.
\end{theorem}

\begin{proof}
It is a consequence of Theorem 1.1  in  \cite{dcfv2}, with $N=1$ and
\[
B(x,t)=\int_{0}^{t}\!b(x,s)\,ds\,.
\]
We recall that in our case for the approximate limits $u^+(x)$ and $u^-(x)$
we have $u^+(x)=\nu_u(x)u(x_+)$ and $u^-(x)=\nu_u(x)u(x_-)$, where $\nu_u=\pm 1$ is the normal at a jump point\,.
\end{proof}

\begin{remark}
When $d=1$, if we assume also all the hypotheses of Theorem \ref{chain} and we compare formulas (\ref{chainrules}) and (\ref{f:int}),  we conclude that
\begin{equation}
\begin{split}
& \int_{-\infty}^{+\infty}\!\!dt\!\!\int_{]a,b[}{\rm sgn}(t)\chi^*_{\Omega_{u,t}}(x)\phi(x)\,dD_x(D_{t}B)(\cdot,t)=
\\  &
\int_{]a,b[} \phi(x) (\nabla_x B)(x,u(x))\, dx
 +
\int_{]a,b[} \phi(x) \psi(x,u(x))\, d\lambda
\\ & +
\sum_{x\in \mathcal N}\!\!\phi(x)
\left[\frac{B(x_{+},u(x_{+}))+B(x_{+},u(x_{-}))}{2}
-\frac{B(x_{-},u(x_{+}))+B(x_{-},u(x_{-}))}{2}\right]\,.
\end{split}
\end{equation}
For piecewise constant functions this formula can be proved by using formula (\ref{asda}) and the following proposition.
\end{remark}

\begin{proposition}\label{chch}
For every piecewise constant function
$u\colon ]a,b[\to\R$ of the form
\[
u(x) = \sum_{i=0}^N v^i \Xi(x),
\]
where $v^0,\ldots, v^N\in\R$,
$a = a_0 < a_1 < \ldots < a_N < a_{N+1} = b$, we have that
$$
\int_{-\infty}^{+\infty}\!\!dt\!\!\int_{]a,b[}{\rm sgn}(t)\chi^*_{\Omega_{u,t}}(x)\phi(x)\,dD_x(D_{t}B)(\cdot,t)=\sum_{i=0}^N \int_{]a,b[} \phi(x)\Xi(x)\, d(D_x B)(\cdot, v^i)\,.
$$
\end{proposition}

\begin{proof}
It is not restrictive to assume that $u\geq 0.$
By the Leibnitz formula (\ref{f:Leib}) we have that
\[
\begin{split}
& \int_{-\infty}^{+\infty}\!\!dt\!\!\int_{]a,b[}{\rm sgn}(t)\chi^*_{\Omega_{u,t}}(x)\phi(x)\,dD_x(D_{t}B)(\cdot,t)
\\ & =
\int_0^{+\infty} dt \int_{]a,b[} \phi(x) d D_x(\chi_{\Omega_{u,t}}(D_{t}B))(\cdot, t)
- \int_0^{+\infty} dt \int_{]a,b[} \phi(x) (D_t B)^*(x,t) d D\chi_{\Omega_{u,t}}(x)
\\ & =: I_1 + I_2\,.
\end{split}
\]
Since $B(\cdot, 0) = 0$, we have that
\begin{equation}\label{f:I1}
\begin{split}
I_1 & = - \int_0^{+\infty} dt \int_{]a,b[} \phi'(x) \chi_{\Omega_{u,t}}(x)\, (D_{t}B)(x, t)\, dx
\\ & =
- \int_{]a,b[} \phi'(x) \left(\int_0^{u(x)} (D_t B)(x,t)\, dt\right)\, dx
\\ & =
- \int_{]a,b[} \phi'(x)\, B(x,u(x))\, dx
\\ & =
- \sum_{i=0}^N \int_{]a,b[} \phi'(x) \chi_{]a_i, a_{i+1}[} (x) B(x,v^i) \, dx
\\ & =
\sum_{i=0}^N\left[
\int_{]a,b[} \phi(x) \chi^*_{[a_i, a_{i+1}]} (x)\, d(D_x B)(x, v^i)
-\phi(a_{i+1}) B^*(a_{i+1}, v^i) + \phi(a_i) B^*(a_i, v^i)
\right].
\end{split}
\end{equation}
For what concerns the second term $I_2$,
let us observe that
$D\chi_{\Omega_{u,t}}$ is an atomic measure with support
contained in $\{a_1, \ldots, a_N\}$.
Moreover
\[
D\chi_{\Omega_{u,t}}(\{a_i\}) =
\begin{cases}
1, &\textrm{if $v_{i-1} < v_{i}$ and $t\in [v_{i-1}, v_{i}[$}, \\
-1, &\textrm{if $v_{i-1} > v_{i}$ and $t\in [v_{i}, v_{i-1}[$}, \\
0, &\textrm{otherwise}\,.
\end{cases}
\]
Therefore, by Fubini's theorem we obtain
\[
I_2 = -\sum_{i=1}^N \phi(a_i) \int_{v_{i-1}}^{v_{i}} (D_t B)^*(a_i,t)\, dt\,.
\]
We claim that
\[
\int_{v_{i-1}}^{v_{i}} (D_t B)^*(a_i,t)\, dt = B^*(a_i, v_{i}) - B^*(a_i, v_{i-1})\,,
\]
so that
\begin{equation}\label{f:I2}
I_2 = -\sum_{i=1}^N \phi(a_i) [B^*(a_i, v_{i}) - B^*(a_i, v_{i-1})]\,.
\end{equation}
Namely, since
$B\left(\cdot, t\right)  \in BV(]a,b[)$, $(D_{t}B)\left(\cdot,t\right)  \in BV(]a,b[)$ and
$|(D_{t} B)(x, t)| \leq D_M$ for every $x\in ]a,b[$,
we have that
\[
\begin{split}
\int_{v_{i-1}}^{v_{i}} & (D_t B)^*(a_i,t)\, dt
=
\frac{1}{2}\int_{v_{i-1}}^{v_{i}}\left[
\lim_{x\to a_i+} (D_t B)(x, t)
+ \lim_{x\to a_i-} (D_t B)(x, t)
\right]\, dt
\\ & =
\lim_{x\to a_i+} \frac{1}{2} \int_{v_{i-1}}^{v_{i}} (D_t B)(x, t)\, dt
+ \lim_{x\to a_i-} \frac{1}{2} \int_{v_{i-1}}^{v_{i}} (D_t B)(x, t)\, dt
\\ & =
\lim_{x\to a_i+} \frac{1}{2} [B(x, v^{i}) - B(x, v^{i-1})]
+ \lim_{x\to a_i-} \frac{1}{2} [B(x, v^{i}) - B(x, v^{i-1})]
\\ & =
B^*(a_i, v_{i}) - B^*(a_i, v_{i-1})\,.
\end{split}
\]
Finally, the conclusion follows from (\ref{f:I1}) and (\ref{f:I2}).
\end{proof}


\section{An application to conservation laws}

In this section we shall apply the chain rule formula in order to study a scalar conservation law where the flux depends discontinuously on the space variable:
\begin{equation}\label{cons}
u_t(x,t)+B(x,u(x,t))_x=0,\qquad (x,t)\in\R\times[0,+\infty)\,,
\end{equation}
where $B:\R\times\R\to\R$ is a function satisfying
the assumptions of Theorem~\ref{chain}
(with $]a,b[ = \R$).

For every $x\in \R$ we define the set of pairs $(u_-,u_+)$ satisfying the Rankine-Hugoniot condition
\[
A_x=\{(u_-,u_+)\in\R\times\R: B(x_-,u_-)=B(x_+,u_+)\}\,.
\]

We define an {\it entropy-flux pair} $(\eta,q)$ associated to \eqref{cons}, as a pair of functions $\eta,q:\R\times\R\rightarrow\R$ such that:
\begin{enumerate}
\item[(E1)] for every $x\in\R$ the function $\eta(x,\cdot)$ is convex and $\eta_u$ is locally bounded in $\R\times\R$;
moreover, for every $u\in\R$ the functions
$\eta(\cdot,u)$, $\eta_u(\cdot,u)$ belong to $BV(\R)$
and their jump set is contained in $J_K$;

\item[(E2)] 
$q(x,\cdot)\in Lip_{\rm loc}(\R)$ for every $x\in\R$,
and $q(\cdot,u)\in BV(\R)$, $J_{q(\cdot,u)}\subseteq J_K$ for every $u\in\R$;

\item[(E3)] $\eta_u(x,u)B_u(x,u)=q_u(x,u)$ for every $x\in\R\setminus J_K$ and $u\in \R$;

\item[(E4)] $q(x_+,u_+)-q(x_-,u_-)\leq 0$ for every $x\in \R$ and every $(u_-,u_+)\in A_x$.
\end{enumerate}

\begin{proposition}\label{bcbc}
Let $u$ be a bounded piecewise $C^1$  solution of \eqref{cons} with 
$J_{u(\cdot,t)}\subseteq J_K$ for every $t\in ]0,T[$,
and let $(\eta,q)$ be an entropy-entropy flux pair  associated to \eqref{cons}.
Then $u$ satisfies the following inequality
\begin{equation}\label{consmeas}
(\eta(x,u))_t+(q(x,u))_x\leq 0
\end{equation}
in the sense of measures, i.e.
\begin{equation}\label{constest}
\int\!\!\!\!\int_{\R\times ]0,T[}\phi(x,t)\,d\,[(\eta(x,u(x,t)))_t+(q(x,u(x,t)))_x]
\leq 0
\end{equation}
for every function $\phi:\R\times ]0,T[\rightarrow[0,\infty[$ continuous with compact support.
\end{proposition}

\begin{proof}
We remark that, by the chain rule formula and since $\eta$ does not depend on $t$, we have
\[
(\eta(x,u(x,t)))_t=\eta^*_u(x,u)u_t(x,t)dt
\]
in the sense of measure. Then
\[
\iint_{\R\times ]0,T[}\phi(x,t)\,d\,(\eta(x,u))_t=
\iint_{\R\times ]0,T[}\phi(x,t)\,\eta^*_u(x,u)u_t(x,t)dt\,,
\]
so that
\begin{align*}
& \iint_{\R\times ]0,T[}\phi(x,t)\,d\,(\eta(x,u))_t+(q(x,u))_x 
\\ & =
\iint_{\R\times]0,T[}\phi(x,t)\,(\eta^*_u(x,u))u_t(x,t)dt+
\iint_{\R\times]0,T[}\phi(x,t)\,d\,(q(x,u))_x\,,
\end{align*}
where $(\eta_u(x,u))^*$ is the precise representative of the composition of $\eta_u(x,\cdot)$ with the function $u$\,.
By (\ref{cons}) we have that
$u_t(x,t)dt=-(B(x,u))_x$
in the sense of measures, i.e.
\[
\iint_{\R\times]0,T[}\phi(x,t)u_t(x,t)dt=-
\iint_{\R\times]0,T[}\phi(x,t)\,d\,(B(x,u))_x\,.
\]
Since the jumps of $\eta_u(\cdot,u))$ are contained in $J_K$, reasoning as in the proof of Corollary \ref{chain222}, we have that
$(\eta_u(x,u))^*u_t(x,t)dt=-(\eta_u(x,u))^*(B(x,u))_x$
in the sense of measures,
i.e.
\[
\iint_{\R\times ]0,T[}\phi(x,t)\,(\eta_u(x,u))^*u_t(x,t)dt=-
\iint_{\R\times]0,T[}\phi(x,t)\,(\eta_u(x,u))^*\,d\,(B(x,u))_x\,.
\]
Hence
\begin{align*}
& \iint_{\R\times ]0,T[}\phi(x,t)\,d\,[(\eta(x,u))_t+(q(x,u))_x]
\\ & = 
-\iint_{\R\times ]0,T[}\phi(x,t)\,(\eta_u(x,u))^*\,d\,(B(x,u))_x+
\iint_{\R\times ]0,T[}\phi(x,t)\,d\,(q(x,u))_x\,,
\end{align*}
so that it is enough to prove that
$(\eta_u(x,u))^* B(x,u)_x\geq(q(x,u))_x$
in the sense of measures, i.e.\ for every nonnegative function
$\phi\in C_c(\R\times ]0,T[)$
\[
\int_0^T\left[\int_\R\phi(x,t)(\eta_u(x,u))^*\,d\,B(x,u)_x\right]dt\geq
\int_0^T\left[\int_\R\phi(x,t)\,d\,q(x,u)_x\right]dt\,.
\]
We use the chain rule formula (see Corollary \ref{chain222}) and condition (E3)
to obtain
\begin{align*}
I  := {} & \int_0^T\left[\int_\R\phi(x,t)(\eta_u(x,u))^*\,d\,B(x,u)_x\right]dt
\\ = {} & 
\int_0^T\left[\int_\R\phi(x,t)\eta_u(x,u)B_u(x,u(x))\,\cdot d\widetilde Du(x)\right]dt
\\ & +\int_0^T\left[\sum_{x\in J_K}\!\!\phi(x,t)
(\eta_u(x,u))^*\Big(B(x_+,u(x_{+}))-B(x_-,u(x_{-}))\Big)
\right]dt\,.
\end{align*}
We remark that the last term vanishes by the Rankine-Hugoniot condition. 
Using (E4) we obtain
\begin{align*}
I \geq {} & \int_0^T\left[\int_\R\phi(x,t)q_u(x,u(x))\,\cdot d\widetilde Du(x)\right]dt
\\ & + 
\int_0^T\Big[\sum_{x\in J_u}\!\!\phi(x,t)
\Big(q(x_+,u(x_+))-q(x_-,u(x_-))\Big)
\Big]dt
\\ = {} &
\int_0^T\left[\int_\R\phi(x,t)\,d\,q(x,u)_x\right]dt\,.
\end{align*}
This concludes the proof.
\end{proof}

We consider the partially adapted Kruzkov entropies introduced 
by Audusse and Perthame for discontinuous flux (see formula (1.3) in \cite{AP}).

In addition to the assumptions on the function $B$ stated in 
Theorem~\ref{chain}, we also assume that
\begin{equation}\label{f:v}
\text{
for every $x\in\R$, the map $B(x,\cdot)$ is a one to one function from $\R$ to $\R$\,.}
\end{equation}

Given $\alpha\in\R$, by assumption (\ref{f:v}) there exists a unique function 
$c_\alpha\colon D_{\alpha}\to\R$, defined on a (possibly empty) set 
$D_{\alpha}\subset\R$, such that
$B(x,c_\alpha(x))=\alpha$ for every $x\in D_{\alpha}$.

\def\etaa{\eta^{(\alpha)}}
\def\qa{q^{(\alpha)}}

\begin{proposition}
For  every $\alpha\in\R$ such that $c_{\alpha}$ is defined in $\R$,
let us define the adapted Kruzkov entropy
\[
\etaa(x,u) :=|u-c_\alpha(x)|
\]
and the corresponding flux
\[
\qa(x,u) :=\big(B(x,u)-\alpha\big)({\rm sgn}(u-c_\alpha(x)))^*\,.
\]
We assume that for every $x\in J_K$ and for every $(u_-,u_+)\in A_x$ we have
\begin{equation}\label{entr}
(\text{sgn}(u_--c_\alpha(x_-)))^*=
(\text{sgn}(u_+-c_\alpha(x_+)))^*\,.
\end{equation}
Then we have
\begin{itemize}
\item(a) $(\etaa,\qa)$ is an entropy-flux pair;
in particular, the entropy inequality
\begin{equation}\label{adapt}
\partial_t|u-c_\alpha(x)|+\partial_x[\big(B(x,u)-\alpha\big)(\text{sgn}(u-c_\alpha(x)))^*]\leq 0
\end{equation}
holds in the sense of distributions;
\item(b) \eqref{consmeas} holds for every entropy-flux pair $(\eta, q)$
if and only if  \eqref{adapt} holds for every $\alpha$ as above.
\end{itemize}
\end{proposition}

\begin{proof}
For every $x\in\R  $ and $u\neq c_{\alpha}(x)$ one has
$\qa_u(x,u)=(\text{sgn}(u-c_\alpha(x))^* B_u(x,u)$.
Then, since $({\etaa_u}(x,u))^*=(\text{sgn}(u-c_\alpha(x)))^*$, we obtain
$\etaa_u(x,u)B_u(x,u)=\qa_u(x,u)$.
Moreover, for every $x\in  \R$ and every $(u_-,u_+)\in A_x$ satisfying (\ref{entr})
we have that
\[
\begin{split}
\qa(x_+,u_+)-\qa(x_-,u_-)
& = 
(\text{sgn}(u_+-c_\alpha(x)))^*\big(B(x_+,u_+)-\alpha\big)
\\ & -
(\text{sgn}(u_--c_\alpha(x)))^*\big(B(x_-,u_-)-\alpha\big)=0\,.
\end{split}
\]

In order to prove (b), let $u$ be a bounded $BV$ solution to (\ref{cons}).
If $u$ satisfies (\ref{consmeas}) for every entropy-flux pair $(\eta, q)$,
then from (a) it satisfies also (\ref{adapt}) for every $\alpha$.

Conversely, assume now that $u$ satisfies also (\ref{adapt}) for every $\alpha$.
Let $(\eta,q)$ be an entropy-flux pair, and
let $\phi(x,t)$ be a non-negative test function.
We have to prove that (\ref{constest}) holds.

Assume that
$|u(x,t)|\leq M$ for every $(x,t)$, $\text{supp} \phi\subset ]a,b[\times ]0,T[$,
and $|B(x, u)|\leq C$ for every $(x,u)\in ]a,b[\times ]-M,M[$.
Let us fix a positive integer number $N$, and for every $x\in [a,b]$ define
\begin{gather*}
I(x) := \{i\in \Z:\ |i| \leq N,\ \alpha^N_i := iC/N \in \text{Range} B(x, \cdot)\}, \\
m(x) := \min I(x),\ n(x) := \max I(x),\\
c^N_i(x) := c_{\alpha^N_i}(x), \ i = m(x), \ldots, n(x).
\end{gather*}
We are going to approximate $\eta$ (and so $q$) by an entropy $\eta^N$ of the
form
\begin{equation}\label{f:etaN}
\begin{split}
\eta^N(x,u) & := a^N(x) + b^N(x) u +
\sum_{i=m(x)+1}^{n(x)-1} b^N_i(x) |u-c^N_i(x)|
\\ & =
a^N(x) + b^N(x) u +
\sum_{i=m(x)+1}^{n(x)-1} b^N_i(x) \eta^{(\alpha^N_i)}(x,u),
\end{split}
\end{equation}
where $b^N_i(x) \geq 0$ for every $i$ and $x$.
Indeed, if we define
\[
\delta^N_i(x) := \frac{\eta(x, c^N_{i+1}(x)) - \eta(x, c^N_i(x))}{c^N_{i+1}(x)- c^N_i(x)}\,,
\quad
x\in \R,\ i=m(x),\ldots,n(x)-1,
\]
and
\begin{gather*}
b^N(x) := \frac{\delta^N_{m(x)}(x) + \delta^N_{n(x)-1}(x)}{2},\\
b^N_i(x) := \frac{\delta^N_{i}(x) - \delta^N_{i-1}(x)}{2}\,,
\ i = m(x)+1,\ldots, n(x)-1,\\
a^N(x) := \eta(x, c^N_{m(x)}(x))-b^N(x) c^N_{m(x)}(x)
- \sum_{i=m(x)+1}^{n(x)-1} b^N_i(x) [c^N_i(x) - c^N_{m(x)}(x)],
\end{gather*}
then $b^N_i(x)\geq 0$ and $\eta^N(x, \cdot)$ is a convex piecewise affine function
coinciding with $\eta(x,\cdot)$ in the points $u_i = c^N_i(x)$, $i=m(x), \ldots, n(x)$.

The flux associated to $\eta^N$ is the function
\begin{equation}\label{f:qN}
q^N(x,u) := b^N(x) B(x,u) +
\sum_{i=m(x)+1}^{n(x)-1} b^N_i(x) q^{(\alpha^N_i)}(x,u)\,.
\end{equation}

For every $N\in\N$ we have that
\begin{equation}\label{dddn}
\begin{split}
& \iint_{\R\times[0,T]}\phi(x,t) [ (\eta^N(x,u(x.t)))_t+ (q^N(x,u(x,t)))_x]
\, dx\;dt
\\ & =
\iint_{\R\times[0,T]} \phi(x,t)\, b^N(x) [u_t(x,t) + (B(x, u(x,t)))_x]\, dx\;dt
\\ & + 
\sum_{i=m(x)+1}^{n(x)-1} \iint_{\R\times[0,T]}  \phi(x,t)\, 
b^N_i(x) [(\eta^{(\alpha^N_i)}(x, u(x,t)))_t + 
(q^{(\alpha^N_i)}(x, u(x,t)))_x] \, dx\;dt \,,
\end{split}
\end{equation}
where we recall that
$$
\eta^{(\alpha^N_i)}(x,u)=|u-c_{\alpha^N_i}(x)|
$$
and
\[
q^{(\alpha^N_i)}(x,u) :=\big(B(x,u)-c_{\alpha^N_i }(x)\big)({\rm sgn}(u-c_{\alpha^N_i}(x))^*\,.
\]

We recall that, given a non-negative measure $\mu$ (i.e. $\mu(\phi)\geq 0$ for every test function $\phi\geq 0$), if we define a measure
$$
\mu_b(\phi):=\iint_{\R\times[0,T]}\phi(x,t) b(x,t)\,d\mu(x,t)\,,
$$
where $b$ is a non-negative Borel function, then $\mu_b$ is also a non-negative measure 
(see \cite[Ch.~7]{Fol})\,.
Hence from (\ref{cons}), (\ref{adapt}) and (\ref{dddn})
and the fact that the functions $b^N_i$ are non-negative, we have
\[
 \iint_{\R\times[0,T]}\phi(x,t) [ (\eta^N(x,u(x.t)))_t+ (q^N(x,u(x,t)))_x]
\, dx\;dt
 \leq 0\,.
\]
The sequences of functions $(\eta^N)_N$, $(q^N)_N$ are uniformly bounded on $[a,b]\times [-M,M]$,
and converge pointwise to $\eta$ and $q$ respectively, hence we conclude that
(\ref{constest}) holds.
\end{proof}



\def\cprime{$'$}

\end{document}